\documentclass[11pt,letterpaper]{article}
\usepackage[utf8]{inputenc}
\usepackage[margin=1in]{geometry}
\usepackage{amsmath, amssymb, amsfonts}
\usepackage{graphicx}
\usepackage{amsthm}
\usepackage{enumitem}
\usepackage{bm}
\usepackage{color}
\usepackage{geometry}
\usepackage{complexity}
\theoremstyle{plain}

\newtheorem{thm}{Theorem}[section]
\newtheorem{lem}[thm]{Lemma}
\newtheorem{clm}[thm]{Claim}

\newtheorem{cor}[thm]{Corollary}

\theoremstyle{definition}
\newtheorem{ex}[thm]{Example}

\makeatletter
\@addtoreset{equation}{section}

\makeatother

\title{A Polynomial Algorithm for Minimizing $k$-Distant Submodular Functions}
\author{Ryuhei Mizutani\thanks{Department of Mathematical Informatics, Graduate School of Information Science and Technology, The University of Tokyo, Tokyo, 113-8656, Japan. E-mail: \texttt{ryuhei\_mizutani@mist.i.u-tokyo.ac.jp}}}

\date{}

\begin{document}

\maketitle

\begin{abstract}
This paper considers the minimization problem of relaxed submodular functions.
For a positive integer $k$, a set function is called $k$-distant submodular if the submodular inequality holds for every pair whose symmetric difference is at least $k$.
This paper provides a polynomial time algorithm to minimize $k$-distant submodular functions for a fixed positive integer $k$.
This result generalizes the tractable result of minimizing 2/3-submodular functions, which satisfy the submodular inequality for at least two pairs formed from every distinct three sets.

\end{abstract}

\section{Introduction}
Let $S$ be a finite set. 
A set function $f:2^S\to \mathbb{R}$ is called \textit{submodular} if the \textit{submodular inequality} $f(X)+f(Y)\ge f(X\cup Y)+f(X\cap Y)$ holds for every pair $X,Y\subseteq S$.
In the submodular function minimization problem, we are given an evaluation oracle of a submodular function $f$, and the task is to find a set $X\subseteq S$ minimizing $f(X)$.
Here, an evaluation oracle computes $f(X)$ for any given $X\subseteq S$.
The first polynomial time algorithm for this problem was given by Gr\"{o}tschel, Lov{\'{a}}sz, and Schrijver \cite{grotschel1981} with the aid of the polynomial solvability equivalence of the separation and optimization.
The first combinatorial strongly polynomial time algorithms were obtained much later by Iwata, Fleischer, and Fujishige \cite{iwata2001} and by Schrijver \cite{schrijver2000}.
For additional improvements in running time, refer to e.g. \cite{chakrabarty2017,iwata2009,lee2015,orlin2009}.

In this paper, we shall introduce a relaxation of submodular functions and consider its minimization problem.
For a positive integer $k$, a set function $f:2^S\rightarrow \mathbb{R}$ is called \textit{$k$-distant submodular} if the submodular inequality holds for every pair $X,Y\subseteq S$ with $|X\bigtriangleup Y|\geq k$, where $X\bigtriangleup Y:=(X\setminus Y)\cup (Y\setminus X)$ is the \textit{symmetric difference} of $X$ and $Y$.
Note that $1$-distant submodular functions and $2$-distant submodular functions are exactly submodular functions because every pair $X,Y\subseteq S$ with $|X\bigtriangleup Y|\le 1$ satisfies $X\subseteq Y$ or $X\supseteq Y$, which implies that such a pair satisfies the submodular inequality.
There are several examples of $k$-distant submodular functions.

\begin{ex}[Cut function with negative edge weight]
For a complete graph $K_n=(V,E)$ and a positive integer $k$, let $w:E\to \mathbb{R}$ be a weight function such that $\sum_{e\in X}w(e)\ge 0$ for every $X\subseteq \delta(v)$ and $v\in V$ with $|X|\ge k$, where $\delta(v)$ is the set of edges incident to $v$.
Then the cut function $f:2^V\to \mathbb{R}$ defined as
\begin{align*}
f(U):=\sum_{e=(u,v)\in E,u\in U,v\in V\setminus U}w(e)
\end{align*}
for every $U\subseteq V$ is $2k-1$-distant submodular.
\end{ex}

\begin{ex}[Perturbation of strictly submodular function]
Let $f:2^S\to \mathbb{Z}$ be an integer-valued \textit{strictly submodular function}, i.e.,  $f(X)+f(Y)>f(X\cup Y)+f(X\cap Y)$ holds for any pair $X,Y\subseteq S$ such that both of $X\setminus Y$ and $Y\setminus X$ are nonempty.
For a positive integer $k$, let $\epsilon:2^S\to \mathbb{R}$ be a set function such that $|\epsilon(X)|\le k$ for every $X\subseteq S$.
Define $f_\epsilon:2^S\to \mathbb{R}$ as $f_\epsilon(X):=f(X)+\epsilon(X)$ for every $X\subseteq S$.
Then $f_\epsilon$ is $4k+1$-distant submodular.
\end{ex}

\begin{ex}[Minimum rank function of two sparse paving matroids]
A matroid $M$ is called \textit{sparse paving} if $M$ and its dual are \textit{paving}, i.e., every circuit has size equal to or greater than the rank of the matroid.
Then the minimum of two rank functions of sparse paving matroids is 3-distant submodular, which arises in the matroid intersection problem under the minimum rank oracle \cite{barasz2006, berczi2024,berczi2023res}.
\end{ex}

\begin{ex}[2/3-Submodular function]
A set function is called \textit{2/3-submodular} \cite{mizutani2024} if it satisfies the submodular inequality for at least two pairs formed from every distinct three subsets.
The class of 2/3-submodular functions is a subclass of 3-distant submodular functions.
\end{ex}

Recently, Mizutani and Yoshida \cite{mizutani2024} showed that 2/3-submodular functions can be minimized in polynomial time with the aid of the ellipsoid method, which yields a polynomial algorithm to obtain a variant of the supermodular coloring \cite{mizutani2022}.
The main result of this paper is a polynomial algorithm to minimize $k$-distant submodular functions, which generalizes the result of \cite{mizutani2024}: 

\begin{thm}
\label{thm:main}
Let $f$ be an integer-valued $k$-distant submodular function on $S$.
A minimizer of $f$ can be found in a polynomial number of arithmetic steps and function evaluations in $|S|^k$ and $\log B$, where $B$ is an upper bound on the absolute values of $f$.
\end{thm}

Note that if $k$ is a fixed positive integer, then integer-valued $k$-distant submodular functions can be minimized in polynomial time by Theorem \ref{thm:main}.

The polynomial algorithm for minimizing submodular functions by Gr\"{o}tschel, Lov{\'{a}}sz, and Schrijver \cite{grotschel1981} reduces the minimization problem to the corresponding linear programming problem over a submodular polyhedron with the aid of the ellipsoid method.
The polynomial algorithm for minimizing 2/3-submodular functions by Mizutani and Yoshida \cite{mizutani2024} employs the same reduction to the linear program over a polyhedron associated with a 2/3-submodular function.
Our algorithm for minimizing $k$-distant submodular functions also uses the same reduction described as follows.
Let $f:2^S\to \mathbb{R}$ be a $k$-distant submodular function and assume $f(\emptyset)=0$ (if $f(\emptyset)\ne 0$, we set $f(X):=f(X)-f(\emptyset)$ for every $X\subseteq S$).
Let $w\in \mathbb{R}_+^S$ be a weight vector and $P(f)=\{x\in \mathbb{R}^S\mid \ x(T)\leq f(T)\ \mathrm{for\ every\ }T\subseteq S\}$, where we use the notation $x(T)=\sum_{s\in T}x(s)$.
To minimize $f$, we consider the following linear program:
\begin{align}
(\mathrm{P})\ \ \ \ \ \mathrm{maximize}&\ \ \ w^\top x\notag\\
\mathrm{subject\ to}&\ \ \ x\in P(f).\notag
\end{align}
If $f$ is integer-valued and $w$ is rational, then by the result of Gr\"{o}tschel, Lov{\'{a}}sz, and Schrijver \cite{grotschel1981} the minimization problem of $f$ can be reduced to (P) in polynomial time with the aid of the ellipsoid method.
Indeed, the separation problem corresponding to (P) includes the problem of deciding whether $f$ is nonnegative or not. 
Hence, by adding the same nonnegative integer to the values of $f$ (except for $f(\emptyset)$) and using the binary search on the separation problem, we can solve the minimization problem of $f$.
Since the separation problem corresponding to (P) can be reduced to the linear program (P) in polynomial time by the ellipsoid method \cite{grotschel1981}, the minimization of $f$ can be reduced to the linear program (P) in time polynomial in $|S|$ and $\log B$.
Hence, in order to prove Theorem \ref{thm:main}, we will show that (P) can be solved in time polynomial in $|S|^k$.

To solve the linear program (P), we investigate the structure of the supports of dual optimal solutions.
This approach is a generalization of the polynomial algorithm by Mizutani and Yoshida~\cite{mizutani2024} for solving the linear program over a polyhedron associated with a 2/3-submodular function.
Indeed, they showed that the linear program over a 2/3-submodular polyhedron has a dual optimal solution whose support has polynomial size by using the uncrossing technique, which led to a polynomial algorithm for solving this  linear program.
Extending this result to the framework of $k$-distant submodular functions, we show that (P) has a dual optimal solution whose support has size polynomial in $|S|^k$ by using the uncrossing technique.
Consequently, this implies that (P) can be solved in time polynomial in $|S|^k$.

\paragraph{Related work.}
There are several relaxed forms of submodularity similar to $k$-distant submodularity.
A set function $f:2^S\to \mathbb{R}$ is called \textit{intersecting submodular} if the submodular inequality holds for every pair $X,Y\subseteq S$ with $|X\cap Y|\ge 1$.
As a slight variant, it is called \textit{crossing submodular} if the submodular inequality holds for every pair $X,Y\subseteq S$ with $|X\cap Y|\ge 1$ and $|X\cup Y|\le |S|-1$.
These types of submodularity are similar to $k$-distant submodularity in the sense that whether the submodular inequality holds relies on the sizes of the results of set operations.
Intersecting or crossing submodular functions can be minimized in polynomial time by reducing the problem to submodular function minimization (see e.g. Chapter 49 in \cite{schrijver2003}).

\paragraph{Organization.}
The rest of this paper is organized as follows.
In Section \ref{sec:pre}, we define several terms and notations, and describe some basic properties including an upper bound on the absolute values of $k$-distant submodular functions.
Section \ref{sec:polyalgo} shows that (P) can be solved in time polynomial in $|S|^k$ with the aid of the structure of the supports of dual optimal solutions.
Section \ref{sec:application} provides two applications of the polynomial algorithm to minimize $k$-distant submodular functions.
The first application is a complexity dichotomy for the minimization problem of $p/q$-submodular functions, which satisfy the submodular inequality for at least $p$ pairs formed from every distinct $q$ sets.
The second application is a tractable result for weighted matroid intersection under the minimum rank oracle \cite{berczi2024,berczi2023res} for some class of matroids.
Section \ref{sec:hardness} proves that $k$-distant submodular function minimization is \textit{W}[1]-hard when parameterized by $k$.
This implies that $k$-distant submodular function minimization with a ground set $S$ has no $g(k)\cdot |S|^{O(1)}$-time algorithm for any computable function $g$.

\section{Preliminaries}
\label{sec:pre}
Throughout this paper, the set of reals, nonnegative reals, integers, nonnegative integers, positive integers, and rationals are denoted by $\mathbb{R},\mathbb{R}_+,\mathbb{Z},\mathbb{Z}_+,\mathbb{N}$, and $\mathbb{Q}$, respectively.
In addition, we assume that for $k$-distant submodular functions on a ground set $S$, $|S|\ge k\ge 2$.

A \textit{matroid} is a pair $\mathbf{M}=(S,\mathcal{I})$ of a ground set $S$ and an \textit{independent set family} $\mathcal{I}$, which satisfies the following three conditions: (i) $\emptyset\in \mathcal{I}$, (ii) $X\subseteq Y\in \mathcal{I}\Rightarrow X\in \mathcal{I}$, (iii) $X,Y\in \mathcal{I}$ and $|X|<|Y|\Rightarrow \exists e\in Y\setminus X$ such that $X\cup \{e\}\in \mathcal{I}$.
A member of $\mathcal{I}$ is called an \textit{independent set}.
The \textit{rank function} $r:2^S\to \mathbb{Z}_+$ of $\mathbf{M}$ is defined as follows: $r(X)=\max\{|I|\mid I\subseteq X,\ I\in \mathcal{I}\}$ for each $X\subseteq S$.
The rank function $r$ is monotone submodular, and satisfies $r(X)\le |X|$ and $r(X\cup \{e\})\le r(X)+1$ for every $X\subseteq S$ and $e\in S$.
For more details on matroids, we refer to \cite{oxley2011,schrijver2003}.

Let $f:2^S\rightarrow \mathbb{R}$ be a $k$-distant submodular function.
Assume that $f(\emptyset)=0$.
We show that an upper bound on the absolute values of $f$ can be found in time polynomial in $|S|^k$.
This follows from the following lemma:
\begin{lem}
\label{lem:upper}
For any $X\subseteq S$, we have
\begin{align*}
f(S)-\sum_{T\subseteq S,|T|\le k}|f(T)|\le f(X)\le \sum_{T\subseteq S,|T|\le k}|f(T)|.
\end{align*}
\end{lem}
\begin{proof}
We first prove
\begin{align}
\label{eq:upper1}
f(S)-\sum_{T\subseteq S,|T|\le k}|f(T)|\le f(X).
\end{align}
Let $P_1,P_2,\ldots,P_l$ be disjoint sets such that $P_1\cup P_2\cup \cdots \cup P_l=S\setminus X,\ |P_i|=k$ for $i=1,\ldots,l-1$, and $|P_l|\le k$.
Then since $f$ is $k$-distant submodular, we have
\begin{align*}
&f(X)+\sum_{T\subseteq S,|T|\le k}|f(T)|\ge f(X)+\sum_{1\le i\le l}f(P_i)=f(X)+f(P_1)+\sum_{2\le i\le l}f(P_i)\\&\ge f(X\cup P_1)+\sum_{2\le i\le l}f(P_i)\ge \cdots \ge f(X\cup P_1\cup \cdots \cup P_{l-1})+f(P_l)\ge f(S).
\end{align*}
The last inequality follows from $|S|\ge k$.
Hence, the inequality (\ref{eq:upper1}) holds.
We next prove
\begin{align}
\label{eq:upper2}
f(X)\le \sum_{T\subseteq S,|T|\le k}|f(T)|.
\end{align}
If $|X|\le k$, then we have 
\begin{align*}
\sum_{T\subseteq S,|T|\le k}|f(T)|\ge |f(X)|\ge f(X),
\end{align*}
as required.
Suppose that $|X|>k$.
Let $Q_1,Q_2,\ldots,Q_m$ be disjoint sets such that $Q_1\cup Q_2\cup \cdots \cup Q_m=X$, $|Q_i|=k$ for $i=1,\ldots,m-1$, and $|Q_m|\le k$.
Then since $f$ is $k$-distant submodular, we have
\begin{align*}
\sum_{T\subseteq S,|T|\le k}|f(T)|&\ge \sum_{1\le i\le m}f(Q_i)\ge f(Q_1\cup Q_2)+\sum_{3\le i\le m}f(Q_i)\ge \cdots \ge f(Q_1\cup \cdots \cup Q_{m-1})+f(Q_m)\\&\ge f(X).
\end{align*}
The last inequality follows from $|X|>k$.
Hence, the inequality (\ref{eq:upper2}) holds, which completes the proof.
\end{proof}

Let $M=\sum_{T\subseteq S,|T|\le k}|f(T)|$.
By Lemma \ref{lem:upper}, $\max\{M,|f(S)-M|\}$ is an upper bound on the absolute values of $f$, and can be computed in time polynomial in $|S|^k$.
For a $k$-distant submodular function $g$ with $g(\emptyset)\ne 0$, we can compute an upper bound on the absolute values of $g-g(\emptyset)$ by Lemma \ref{lem:upper}, which implies that an upper bound on the absolute values of $g$ can be computed in time polynomial in $|S|^k$.

\section{A polynomial algorithm}
\label{sec:polyalgo}
In this section, we prove that an optimal solution of the linear program $(\mathrm{P})$ can be computed in time polynomial in $|S|^k$.
Recall that $(\mathrm{P})$ is defined as follows for a $k$-distant submodular function $f:2^S\rightarrow \mathbb{R}$ and a weight vector $w\in \mathbb{R}^S_+$:
\begin{align}
(\mathrm{P})\ \ \ \ \ \mathrm{maximize}&\ \ \ w^\top x\notag\\
\mathrm{subject\ to}&\ \ \ x\in P(f),\notag
\end{align}
where $P(f)=\{x\in \mathbb{R}^S\mid \ x(T)\leq f(T)\ \mathrm{for\ every\ }T\subseteq S\}$.
We first show that the dual linear program of $(\mathrm{P})$ has an optimal solution whose support is contained in a family of size polynomial in $|S|^k$.
Then we prove that $(\mathrm{P})$ can be reduced to a linear program whose constraints are restricted to that family.

Consider the following dual linear program of $(\mathrm{P})$:
\begin{align}
(\mathrm{D})\ \ \ \ \ \mathrm{minimize}&\ \ \ \sum_{T\subseteq S}y_T f(T)\notag\\
    \mathrm{subject\ to}&\ \ \ y\in \mathbb{R}_+^{2^S},\notag\\&\ \ \ \sum_{T\subseteq S}y_T\chi_T=w.\notag
\end{align}
Here, $\chi_T$ denotes the \textit{incidence vector} of a subset $T\subseteq S$ defined by
\begin{align*}
    \chi_T(s):=\begin{cases}
        1 & \mathrm{if\ }s\in T,\\
        0 & \mathrm{if\ }s\in S\setminus T.
    \end{cases}
\end{align*}
Order the elements of $S$ as $s_1,s_2,\ldots,s_n$ such that $w(s_1)\geq w(s_2)\geq \cdots\geq w(s_n)$. 
Define $U_i=\{s_1,\ldots,s_i\}$ for each $i=1,\ldots,n$, and $U_0=\emptyset$. 
Let $\mathcal{C}=\{U_i\bigtriangleup T\mid 0\leq i\leq n,\ T\subseteq S,\ |T|\le k-2\}$.

\begin{lem}
\label{lem:chaink}
There exists an optimal solution $y$ of $(\mathrm{D})$ such that its support $\mathcal{L}=\{T\subseteq S\mid y_T>0\}$ is included in $\mathcal{C}$.
\end{lem}
\begin{proof}
Let $y$ be an optimal solution of $(\mathrm{D})$ minimizing
\begin{align}
\sum_{T\subseteq S}y_T|T||S\setminus T|.\notag
\end{align}
Then, the support $\mathcal{L}$ of $y$ is close to being a chain in some sense.
\begin{clm}
\label{clm:supportk}
Every pair of sets $X,Y\in \mathcal{L}$ satisfies one of the following two conditions:
\begin{itemize}
    \item $X$ and $Y$ are comparable.
    \item $|X\bigtriangleup Y|\le k-1$.
\end{itemize}
\end{clm}
\begin{proof}[Proof of Claim \ref{clm:supportk}]
Take a pair of sets $X,Y\in \mathcal{L}$. Suppose to the contrary that $X$ and $Y$ are incomparable, and $|X\bigtriangleup Y|\ge k$. 
Let $\alpha=\min\{y_X,y_Y\}$.
Then decrease $y_X$ and $y_Y$ by $\alpha$ and increase $y_{X\cup Y}$ and $y_{X\cap Y}$ by $\alpha$. 
Let $y'$ denote the vector obtained by this way. Then $y'$ is a feasible solution of $(\mathrm{D})$. 
Since $|X\bigtriangleup Y|\geq k$, we have $f(X)+f(Y)\geq f(X\cup Y)+f(X\cap Y)$, which implies that
\begin{align}
\sum_{T\subseteq S}y'_T f(T)\leq \sum_{T\subseteq S}y_T f(T).\notag
\end{align}
Hence, $y'$ is an optimal solution of $(\mathrm{D})$.
Since $X$ and $Y$ are incomparable, we have
\begin{align}
\sum_{T\subseteq S}y'_T|T||S\setminus T|<\sum_{T\subseteq S}y_T|T||S\setminus T|,\notag
\end{align}
which is a contradiction.
\end{proof}

Suppose to the contrary that $\mathcal{L}\not\subseteq \mathcal{C}$.
Take $V\in \mathcal{L}\setminus \mathcal{C}$.
Since $V=U_i\bigtriangleup (V\bigtriangleup U_i)$ for $i=0,\ldots,n$ and $V\notin \mathcal{C}$, we have $\min\{|V\bigtriangleup U_i|\mid 0\leq i\leq n\}\geq k-1$.
Note that $|V\bigtriangleup U_0| =|V|\geq k-1$ and $|V\bigtriangleup U_n| =|S\setminus V|\geq k-1$. 
Let $i_p$ be the $p$th largest integer such that $s_{i_p}\in V$ for $p=1,\ldots,k-1$.
Similarly, let $j_{k-p}$ be the $p$th smallest integer such that $s_{j_{k-p}}\notin V$ for $p=1,\ldots,k-1$.
Define 
\begin{align*}
&\rho_1(T)=|\{p\mid 1\le p\le k-1,s_{i_p}\notin T,s_{j_p}\in T\}|,\\
&\rho_2(T)=|\{p\mid 1\le p\le k-1,s_{i_p}\in T,s_{j_p}\notin T\}|,\\
&\rho_3(T)=|\{p\mid 1\le p\le k-1,s_{i_p},s_{j_p}\in T\mathrm{\ or\ }s_{i_p},s_{j_p}\notin T\}|,\\
& \rho(T)=\rho_1(T)-\rho_2(T)
\end{align*}
for each $T\in \mathcal{L}$.
Then we show $\rho(T)\le 0$ for any $T\in \mathcal{L}$ as follows.
Since $T,V\in \mathcal{L}$, it holds by Claim~\ref{clm:supportk} that $T$ and $V$ are comparable, or $|T\bigtriangleup V|\le k-1$.
If $T$ and $V$ are comparable, then since $s_{i_p}\in V$ and $s_{j_p}\notin V$ for $p=1,\ldots,k-1$, both $s_{i_p}\notin T$ and $s_{j_p}\in T$ cannot hold at the same time for each $p=1,\ldots,k-1$, which implies $\rho(T)=-\rho_2(T)\le 0$.
Consider the case of $|T\bigtriangleup V|\le k-1$.
Since $s_{i_p}\in V$ and $s_{j_p}\notin V$ for $p=1,\ldots,k-1$, we have $2\cdot \rho_1(T)+\rho_3(T)\le |T\bigtriangleup V|\le k-1$.
Hence, since $\rho_1(T)+\rho_2(T)+\rho_3(T)=k-1$, we have
\begin{align*}
\rho(T)=\rho_1(T)-\rho_2(T)=2\cdot \rho_1(T)+\rho_3(T)-k+1\le 0.
\end{align*}
Thus, $\rho(T)\le 0$ holds for every $T\in \mathcal{L}$.
Since $y$ is a feasible solution of $(\mathrm{D})$, we have
\begin{align*}
\sum_{T\in \mathcal{L}}y_T\chi_T=w.
\end{align*}
Hence, the following holds for $p=1,\ldots,k-1$:
\begin{align*}
w(s_{j_p})-w(s_{i_p})=\sum_{\substack{T\in \mathcal{L}\\s_{i_p}\notin T,s_{j_p}\in T}}y_T-\sum_{\substack{T\in \mathcal{L}\\s_{i_p}\in T,s_{j_p}\notin T}}y_T.
\end{align*}
This implies
\begin{align}
\label{eq:wdiff}
\sum_{1\le p\le k-1}(w(s_{j_p})-w(s_{i_p}))=\sum_{\substack{T\in \mathcal{L}\\s_{i_p}\notin T,s_{j_p}\in T\\1\le p\le k-1}}y_T-\sum_{\substack{T\in \mathcal{L}\\s_{i_p}\in T,s_{j_p}\notin T\\1\le p\le k-1}}y_T&=\sum_{T\in \mathcal{L}}y_T\rho_1(T)-\sum_{T\in \mathcal{L}}y_T\rho_2(T)\notag\\&=\sum_{T\in \mathcal{L}}y_T\rho(T)<0.
\end{align}
The last inequality follows from $y_T\rho(T)\le 0$ for every $T\in \mathcal{L}$ and $y_V\rho(V)< 0$.
If $i_p>j_p$ for $p=1,\ldots,k-1$, then $w(s_{j_p})\ge w(s_{i_p})$ holds for every $p=1,\ldots,k-1$, contradicting (\ref{eq:wdiff}).
Hence, $i_p<j_p$ holds for some $1\le p\le k-1$.
Then we have 
\begin{align*}
&U_{i_p}\setminus V\subseteq \{s_{j_{p+1}},s_{j_{p+2}},\ldots,s_{j_{k-1}}\},\\
&V\setminus U_{i_p}= \{s_{i_1},s_{i_2},\ldots,s_{i_{p-1}}\}.
\end{align*}
This implies $|V\bigtriangleup U_{i_p}|\le k-2$, contradicting $\min\{|V\bigtriangleup U_i|\mid 0\leq i\leq n\}\geq k-1$.
\end{proof}

Consider the following linear program restricted to $\mathcal{C}$:
\begin{align}
(\mathrm{D}')\ \ \ \ \ \mathrm{minimize}&\ \ \ \sum_{T\in \mathcal{C}}y_T f(T)\notag\\
    \mathrm{subject\ to}&\ \ \ y\in \mathbb{R}_+^{\mathcal{C}},\notag\\&\ \ \ \sum_{T\in \mathcal{C}}y_T\chi_T=w.\notag
\end{align}
The following corollary immediately follows from Lemma \ref{lem:chaink}.
\begin{cor}
\label{cor:restrictedc}
Let $y\in \mathbb{R}_+^{\mathcal{C}}$ be an optimal solution of $(\mathrm{D}')$.
Define $z\in \mathbb{R}_+^{2^S}$ as $z_T=y_T$ for $T\in \mathcal{C}$, and $z_T=0$ for $T\in 2^S\setminus \mathcal{C}$.
Then $z$ is an optimal solution of $(\mathrm{D})$.
\end{cor}
Consider the following linear program, which is the dual of $(\mathrm{D}')$:
\begin{align}
(\mathrm{P}')\ \ \ \ \ \mathrm{maximize}&\ \ \ w^\top x\notag\\
\mathrm{subject\ to}&\ \ \ x\in P_{\mathcal{C}}(f),\notag
\end{align}
where $P_{\mathcal{C}}(f)=\{x\in \mathbb{R}^S\mid \ x(T)\leq f(T)\ \mathrm{for\ every\ }T\in \mathcal{C}\}$. 
Let $x^*$ be an extreme point optimal solution of $(\mathrm{P}')$.
We show the following key lemma:

\begin{lem}
\label{lem:extreme_monotonek}
Suppose that $w(s_1)>w(s_2)>\cdots >w(s_n)$.
Then $x^*\in P(f)$.
\end{lem}
\begin{proof}
Let $y$ be an optimal solution of $(\mathrm{D}')$.
Let $\mathcal{F}=\{T\in \mathcal{C}\mid x^*(T)=f(T)\}$ and $\mathcal{L}=\{T\in \mathcal{C}\mid y_T>0\}$.
By complementary slackness, we have $\mathcal{L}\subseteq \mathcal{F}$.
Define $\delta=\min\{w(s_i)-w(s_{i+1})\mid 1\le i\le n-1\}$.
Note that $\delta>0$ because $w(s_1)>w(s_2)>\cdots >w(s_n)$.
Let $y^*\in \mathbb{R}^{\mathcal{C}}_+$ be a vector defined as follows:
\begin{align*}
y^*_T=\left\{\begin{array}{lll}
    y_T & \mathrm{if\ }T\in \mathcal{L},\\
    \dfrac{\delta}{(n+1)\cdot n^k} & \mathrm{if\ }T\in \mathcal{F}\setminus \mathcal{L},\\
    0 & \mathrm{if\ }T\in \mathcal{C}\setminus \mathcal{F}.
\end{array} \right.
\end{align*}
Let $w^*=\sum_{T\in \mathcal{C}}y^*_T\chi_T$.
Recall that we assume $k\le |S|=n$.
Since
\begin{align*}
|\mathcal{C}|\le (n+1)\cdot \sum_{0\le i\le k-2}\binom{n}{i}\le (n+1)\cdot n^{k-2}\cdot (k-1)\le (n+1)\cdot n^{k-1}
\end{align*}
and $\delta>0$, we have
\begin{align*}
w^*(s_i)-w^*(s_{i+1})\ge w(s_i)-w(s_{i+1})-\dfrac{\delta}{(n+1)\cdot n^k}\cdot |\mathcal{C}|\ge \delta-\dfrac{\delta}{n}>0 
\end{align*}
for each $i=1,\ldots,n-1$.
Hence, we have $w^*(s_1)>w^*(s_2)>\cdots >w^*(s_n)$.
Let $(\mathrm{P}^*)$ and $(\mathrm{D}^*)$ be linear programs obtained by replacing $w$ with $w^*$ in $(\mathrm{P}')$ and $(\mathrm{D}')$, respectively.
By complementary slackness, $x^*$ and $y^*$ are optimal solutions of $(\mathrm{P}^*)$ and $(\mathrm{D}^*)$, respectively.
Define $z\in \mathbb{R}^{2^S}_+$ as $z_T=y^*_T$ if $T\in \mathcal{C}$, and $z_T=0$ otherwise.
Then by Corollary \ref{cor:restrictedc}, $z$ is also an optimal solution of $(\mathrm{D})$ with $w=w^*$.
Hence, by complementary slackness any optimal solution $x$ of $(\mathrm{P})$ with $w=w^*$ satisfies $x(T)=f(T)$ for every $T\in \mathcal{F}$.
This implies $x^*=x\in P(f)$.
\end{proof}

By Lemma \ref{lem:extreme_monotonek}, we obtain the following theorem:
\begin{thm}
\label{thm:kpolynomial}
Suppose that $f$ is integer-valued and $w$ is rational. Then the linear program $(\mathrm{P})$ can be solved in time polynomial in $|S|^k$.
\end{thm}
\begin{proof}

Let $B_w$ be an upper bound on the denominators of the absolute values of $w(s_1),\ldots,w(s_n)$.
Let $B_f$ be an upper bound on the absolute values of $f$.
Define $\epsilon>0$ as follows:
\begin{align*}
\epsilon=\dfrac{1}{4n^2}\cdot \left(\dfrac{1}{n!}\right)^3\cdot \left(\dfrac{1}{B_w}\right)^{n}\cdot \dfrac{1}{B_f}.
\end{align*}
Define $w_\epsilon\in \mathbb{R}^S_+$ as follows:
\begin{align*}
w_\epsilon(s_i)=\begin{cases}
    w(s_i)+\epsilon\cdot (n-i) & \mathrm{if\ }1\le i\le n-1\ \mathrm{and\ }w(s_i)=w(s_{i+1}),\\
    w(s_i) & \mathrm{otherwise}.
\end{cases}
\end{align*}
Then, we have $w_\epsilon(s_i)>w_\epsilon(s_{i+1})$ for $i=1,\ldots,n-1$.
Indeed, if $w_\epsilon(s_i)=w(s_i)+\epsilon\cdot (n-i)$, then $w_\epsilon(s_i)=w(s_i)+\epsilon\cdot (n-i)>w(s_{i+1})+\epsilon\cdot (n-i-1)\ge w_\epsilon(s_{i+1})$, as required.
Consider the case when $w_\epsilon(s_i)=w(s_i)$.
Since $w(s_i)>w(s_{i+1})$, we have
\begin{align*}
w(s_i)\ge w(s_{i+1})+\dfrac{1}{{B_w}^2}.
\end{align*}
Hence,
\begin{align*}
w_\epsilon(s_i)=w(s_i)\ge w(s_{i+1})+\dfrac{1}{{B_w}^2}> w(s_{i+1})+\epsilon\cdot (n-i-1)\ge w_\epsilon(s_{i+1}),
\end{align*}
as required.

Let $x^*$ be an extreme point optimal solution of $(\mathrm{P}')$ with $w=w_\epsilon$.
Since $w_\epsilon(s_1)>w_\epsilon(s_2)>\cdots >w_\epsilon(s_n)$, $x^*$ is also an optimal solution of $(\mathrm{P})$ with $w=w_\epsilon$ by Lemma \ref{lem:extreme_monotonek}.
Suppose that $x^*$ is further an optimal solution of the original linear program $(\mathrm{P})$.
Since $|\mathcal{C}|\le 2n^k$ by the proof of Lemma \ref{lem:extreme_monotonek}, $x^*$ can be computed in time polynomial in $|S|^k$ by the ellipsoid method in \cite{tardos1986}.
Hence, $(\mathrm{P})$ can also be solved in time polynomial in $|S|^k$ by computing $x^*$.
Therefore, it remains to show that $x^*$ is an optimal solution of $(\mathrm{P})$.
Let $x_1,\ldots,x_m$ be extreme point optimal solutions of $(\mathrm{P})$, and $x'$ another extreme point of $P(f)$.
For $i=1,\ldots,m$ and $j=1,\ldots,n$, by Cramer's rule we have
\begin{align}
\label{eq:cramer}
w_\epsilon(s_j)(x_i(s_j)-x'(s_j))&\ge w(s_j)(x_i(s_j)-x'(s_j))-\epsilon n\cdot \left|x_i(s_j)-x'(s_j)\right|\notag\\&\ge w(s_j)(x_i(s_j)-x'(s_j))-\epsilon n\cdot \left|x_i(s_j)\right|-\epsilon n\cdot \left|x'(s_j)\right|\notag\\&\ge w(s_j)(x_i(s_j)-x'(s_j))-2\epsilon n\cdot n!\cdot B_f.
\end{align}
Since $x'$ is not an optimal solution of $(\mathrm{P})$, we have $w^{\top}x_i>w^{\top}x'$ for $i=1,\ldots,m$.
Hence, by Cramer's rule the following holds for $i=1,\ldots,m$:
\begin{align}
\label{eq:cramer2}
w^{\top}x_i-w^{\top}x'\ge \left(\dfrac{1}{n!}\right)^2\cdot  \left(\dfrac{1}{B_w}\right)^{n}.
\end{align}
By (\ref{eq:cramer}) and (\ref{eq:cramer2}), for each $i=1,\ldots,m$ we have
\begin{align*}
w_\epsilon^{\top}x_i-w_\epsilon^{\top}x'&\ge w^{\top}x_i-w^{\top}x'-2\epsilon n^2\cdot n!\cdot B_f\\&\ge \left(\dfrac{1}{n!}\right)^2\cdot  \left(\dfrac{1}{B_w}\right)^{n}-2\epsilon n^2\cdot n!\cdot B_f\\&=\left(\dfrac{1}{n!}\right)^2\cdot \left(\dfrac{1}{B_w}\right)^{n}-\dfrac{1}{2}\cdot \left(\dfrac{1}{n!}\right)^2\cdot\left(\dfrac{1}{B_w}\right)^{n}\\&=\dfrac{1}{2}\cdot \left(\dfrac{1}{n!}\right)^2\cdot\left(\dfrac{1}{B_w}\right)^{n}>0.
\end{align*}
This implies that every extreme point optimal solution of $(\mathrm{P})$ with $w=w_\epsilon$ is one of $x_1,\ldots,x_m$.
Hence, $x^*=x_i$ holds for some $1\le i\le m$.
Therefore, $x^*$ is an extreme point optimal solution of $(\mathrm{P})$.
\end{proof}

\section{Applications}
\label{sec:application}
In this section, we present two applications of our result.
The first one is a complexity dichotomy of the $p/q$-submodular function minimization problem \cite{mizutani2024,berczi2008} for general integers $p$ and $q$.
The second one is a polynomial algorithm for the weighted matroid intersection problem under the minimum rank oracle \cite{berczi2024,berczi2023res} for some class of matroids.

\subsection{A complexity dichotomy for $p/q$-submodular function minimization}

Let $S$ be a finite set and $p,q$ fixed positive integers with $p\le \tbinom{q}{2}$.
A set function $f:2^S\to \mathbb{R}$ is called \textit{$p/q$-submodular} if for every distinct $q$ subsets $S_1,\ldots,S_q\subseteq S$, there exist at least $p$ pairs $(i_1,j_1),\ldots,(i_p,j_p)\in \{(i,j)\mid 1\le i<j\le q\}$ such that
\begin{align*}
f(S_{i_m})+f(S_{j_m})\ge f(S_{i_m}\cup S_{j_m})+f(S_{i_m}\cap S_{j_m})
\end{align*}
for $m=1,\ldots,p$.
Note that for the case of $q=2$, 1/2-submodular functions exactly coincide with the class of submodular functions, and can be minimized in polynomial time \cite{grotschel1981,iwata2001,schrijver2000}.
For the case of $q=3$, Mizutani and Yoshida \cite{mizutani2024} showed that 2/3-submodular functions can be minimized in polynomial time in the value oracle model, while B{\'{e}}rczi and Frank \cite{berczi2008} proved that it requires an exponential number of evaluation oracle calls to minimize 1/3-submodular functions.
We will provide a complexity dichotomy theorem for minimizing $p/q$-submodular functions for general $q\ge 3$ with the aid of Theorem \ref{thm:main}.
We first prove the following lemma:
\begin{lem}
\label{lem:pq-submo}
For a positive integer $q\ge 3$, any $p/q$-submodular function is $2q-3$-distant submodular if $p\ge \tbinom{q-1}{2}+1$.
\end{lem}
\begin{proof}
Let $f:2^S\to \mathbb{R}$ be a $p/q$-submodular function.
Take $X,Y\subseteq S$ with $|X\bigtriangleup Y|\ge 2q-3$.
Let $X\setminus Y=\{x_1,\ldots,x_l\}$ and $Y\setminus X=\{y_1,\ldots,y_m\}$, where $l=|X\setminus Y|$ and $m=|Y\setminus X|$.
Let $X_i=(X\cap Y)\cup \{x_i\}$ for $i=1,\ldots,l$, and $X_i=X\cup \{y_{i-l}\}$ for $i=l+1,\ldots,l+m$.
By the definition of $X_i$, we have $X_i\ne X_j$ and $Y\ne X_i$ for any $1\le i\ne j\le l+m$.
Choose any $q-1$ sets $X_{i_1},\dots,X_{i_{q-1}}$ from $X_1,\ldots,X_{l+m}$.
Then $X_{i_1},\dots,X_{i_{q-1}},Y$ are distinct $q$ sets.
Since $f$ is $p/q$-submodular and $p\ge \tbinom{q-1}{2}+1$, we have $f(X_{i_j})+f(Y)\ge f(X_{i_j}\cup Y)+f(X_{i_j}\cap Y)$ for some $1\le j\le q-1$.
Hence, the number of sets $X_i$ satisfying $f(X_i)+f(Y)<f(X_i\cup Y)+f(X_i\cap Y)$ is at most $q-2$.
This implies that $f(X_i)+f(Y)\ge f(X_i\cup Y)+f(X_i\cap Y)$ holds for at least $l+m-q+2$ sets of $X_i$.
Let $Y_i=Y\cup \{x_i\}$ for $i=1,\ldots,l$, and $Y_i=(X\cap Y)\cup \{y_{i-l}\}$ for $i=l+1,\ldots,l+m$.
By the definition of $Y_i$, we have $Y_i\ne Y_j$ and $X\ne Y_i$ for any $1\le i\ne j\le l+m$.
Hence, by the same argument as above, we have $f(X)+f(Y_i)\ge f(X\cup Y_i)+f(X\cap Y_i)$ for at least $l+m-q+2$ sets of $Y_i$.
Since $l+m\ge 2q-3$ by the definition of $l$ and $m$, we have $2(l+m-q+2)\ge l+m+2q-3-2q+4>l+m$.
Hence, for some $1\le i\le l+m$ we have
\begin{align}
\label{eq:pq}
&f(X_i)+f(Y)\ge f(X_i\cup Y)+f(X_i\cap Y)\ \mathrm{and}\notag\\
&f(X)+f(Y_i)\ge f(X\cup Y_i)+f(X\cap Y_i).
\end{align}
If $1\le i\le l$, then (\ref{eq:pq}) implies
\begin{align*}
&f((X\cap Y)\cup \{x_i\})+f(Y)\ge f(Y\cup \{x_i\})+f(X\cap Y)\ \mathrm{and}\\
&f(X)+f(Y\cup \{x_i\})\ge f(X\cup Y)+f((X\cap Y)\cup \{x_i\}).
\end{align*}
By adding these inequalities, we obtain $f(X)+f(Y)\ge f(X\cup Y)+f(X\cap Y)$, as desired.
If $l+1\le i\le l+m$, then (\ref{eq:pq}) implies
\begin{align*}
&f(X\cup \{y_{i-l}\})+f(Y)\ge f(X\cup Y)+f((X\cap Y)\cup \{y_{i-l}\})\ \mathrm{and}\\
&f(X)+f((X\cap Y)\cup \{y_{i-l}\})\ge f(X\cup \{y_{i-l}\})+f(X\cap Y).
\end{align*}
By adding these inequalities, we obtain $f(X)+f(Y)\ge f(X\cup Y)+f(X\cap Y)$, as desired.
\end{proof}

By Lemma \ref{lem:pq-submo} and Theorem \ref{thm:main}, integer-valued $p/q$-submodular functions can be minimized in polynomial time for fixed positive integers $p,q$ with $q\ge 3$ and $p\ge \tbinom{q-1}{2}+1$.
To show the hardness of the other cases, we define a set function $f_T:2^S\to \mathbb{Z}$ for each $T\subseteq S$ as follows:
\begin{align*}
f_T(X)=\begin{cases}
    -1 & \mathrm{if}\ X=T,\\
    0 & \mathrm{otherwise}.
\end{cases}
\end{align*}
Since $f_T$ is $p/q$-submodular for positive integers $p,q$ with $q\ge 3$ and $p\le \tbinom{q-1}{2}$, and since it requires an exponential number of evaluation oracle queries to minimize $f_T$, the minimization of $p/q$-submodular functions with $q\ge 3$ and $p\le \tbinom{q-1}{2}$ requires an exponential number of evaluation oracle calls.
Combining these results, we obtain the following theorem:
\begin{thm}
\label{thm:dichotomy}
Let $p$ and $q$ be fixed positive integers with $q\ge 3$ and $p\le \tbinom{q}{2}$. 
Then integer-valued $p/q$-submodular functions can be minimized in polynomial time if $p\ge \tbinom{q-1}{2}+1$.
Otherwise, it requires an exponential number of evaluation oracle calls for the minimization.
\end{thm}

\subsection{Weighted matroid intersection under the minimum rank oracle}
For $i=1,2$, let $\mathbf{M}_i=(S,\mathcal{I}_i)$ be matroids on the same ground set $S$, whose independent set families and rank functions are denoted by $\mathcal{I}_i$ and $r_i$, respectively. 
The \textit{minimum rank function} of $\mathbf{M}_1$ and $\mathbf{M}_2$ is a set function $r_{\min}:2^S\to \mathbb{Z}$ defined as $r_{\min}(X):=\min\{r_1(X),r_2(X)\}$ for each $X\subseteq S$. 
B{\'{a}}r{\'{a}}sz et al. \cite{barasz2006,berczi2024} showed that a maximum cardinality common independent set of two matroids can be obtained in polynomial time under the \textit{minimum rank oracle}, which returns only $r_{\min}(X)$ instead of $r_1(X)$ and $r_2(X)$ for each $X\subseteq S$.
On the other hand, the weighted matroid intersection problem under the minimum rank oracle has been solved in polynomial time for only few cases including the case when one of the matroids is an elementary split matroid \cite{berczi2023res}, and the case when no circuit in one matroid includes any circuit in the other \cite{berczi2024}, and the case when the maximum size of a circuit in one matroid is bounded by a constant \cite{berczi2024}.
We will prove by Theorem \ref{thm:main} that the weighted matroid intersection problem under the minimum rank oracle is tractable when both of the matroids are close to uniform in some sense.

Suppose that $r_1(S)=r_2(S)$ and let $r:=r_1(S)=r_2(S)$.
For a fixed positive integer $k\le r$, the following lemma holds:

\begin{lem}
\label{lem:min_rank}
For every $X\subseteq S$ and $i=1,2$, suppose that $r_i(X)=|X|$ holds if $|X|\le r-k$, and $r_i(X)=r$ holds if $|X|\ge r+k$.
Then $r_{\min}$ is $4k$-distant submodular.
\end{lem}

\begin{proof}
For $X,Y\subseteq S$, suppose that $|X\bigtriangleup Y|\ge 4k$.
We will show that $r_{\min}(X)+r_{\min}(Y)\ge r_{\min}(X\cup Y)+r_{\min}(X\cap Y)$.
If $|X|\ge r+k$, then 
\begin{align*}
r_{\min}(X)+r_{\min}(Y)=r+r_{\min}(Y)=r+r_{i}(Y)&\ge r_{\min}(X\cup Y)+r_i(X\cap Y)\\&\ge r_{\min}(X\cup Y)+r_{\min}(X\cap Y),
\end{align*}
where $i$ is either $1$ or $2$. 
Hence, it suffices to consider the case when $|X|,|Y|<r+k$.
If $|X|,|Y|\le r-k$, then $r_{\min}(X)+r_{\min}(Y)=|X|+|Y|=|X\cup Y|+|X\cap Y|\ge r_{\min}(X\cup Y)+r_{\min}(X\cap Y)$, as desired.
If $|X|\le r-k<|Y|<r+k$, then 
\begin{align*}
r_{\min}(X)+r_{\min}(Y)&=|X|+r_{\min}(Y)\\&=|X|-r_{\min}(X\cap Y)+r_{\min}(Y)-r_{\min}(X\cup Y)+r_{\min}(X\cup Y)+r_{\min}(X\cap Y)\\&=|X\setminus Y|+r_{\min}(Y)-r_{\min}(X\cup Y)+r_{\min}(X\cup Y)+r_{\min}(X\cap Y)\\&\ge |X\setminus Y|+r_{i}(Y)-r_{i}(X\cup Y)+r_{\min}(X\cup Y)+r_{\min}(X\cap Y)\\&\ge r_{\min}(X\cup Y)+r_{\min}(X\cap Y),
\end{align*}
where $i\in \{1,2\}$ satisfies $r_{\min}(Y)=r_i(Y)$.
Consider the case when $r-k<|X|,|Y|<r+k$.
Since $|X\bigtriangleup Y|\ge 4k$, we may assume without loss of generality that $|X\setminus Y|\ge 2k$.
Take $X\subseteq Z\subseteq S$ with $|Z|\le r+k$ and $r_{\min}(Z)=r$.
Since $|X|= |X\cap Y|+|X\setminus Y|\ge r_{\min}(X\cap Y)+2k$, we have 
\begin{align*}
r_{\min}(X)\ge r_{\min}(Z)-|Z\setminus X|\ge r-(r+k)+|X|= |X|-k\ge r_{\min}(X\cap Y)+k.
\end{align*}
Hence,
\begin{align*}
r_{\min}(X)+r_{\min}(Y)\ge r_{\min}(X\cap Y)+k+r-k&=r+r_{\min}(X\cap Y)\\&\ge r_{\min}(X\cup Y)+r_{\min}(X\cap Y).
\end{align*}
\end{proof}

The class of matroids that satisfy the conditions in Lemma \ref{lem:min_rank} is a subclass of $k$-paving matroids \cite{rajpal1998}, and exactly the class of sparse paving matroids if $k=1$.
Here, a matroid is called \textit{$k$-paving} if all circuits have cardinality more than $r-k$.

Let $w\in \mathbb{Z}^S$ be a weight vector.
Consider the following linear programming formulation of the weighted matroid intersection under the minimum rank oracle:
\begin{align}
(\mathrm{MI})\ \ \ \ \ \mathrm{maximize}&\ \ \ w^\top x\notag\\
    \mathrm{subject\ to}&\ \ \ x\in \mathbb{R}_+^{S},\notag\\&\ \ \ x(T)\le r_{\min}(T),\ \ \ \forall T\subseteq S.\notag
\end{align}
Since the intersection of independent set polytopes of two matroids is the convex hull of the incidence vectors of the common independent sets \cite{edmonds1970}, an extreme point optimal solution of $(\mathrm{MI})$ corresponds to a maximum weight common independent set.
By the polynomial solvability equivalence of the separation and optimization \cite{grotschel1981}, $(\mathrm{MI})$ can be polynomially reduced to the corresponding separation problem.
If $r_{\min}$ is $4k$-distant submodular, then since $f(T):=r_{\min}(T)-x(T)$ for $T\subseteq S$ is also $4k$-distant submodular for any $x\in \mathbb{Q}^S$, the corresponding separation problem can be solved in time polynomial in $|S|^k$ and $\log B$ by Theorem \ref{thm:main}, where $B$ is an upper bound on the absolute values of $w(s)$ for $s\in S$.
Hence, by Lemma \ref{lem:min_rank} we obtain the following theorem:

\begin{thm}
For every $X\subseteq S$ and $i=1,2$, suppose that $r_i(X)=|X|$ holds if $|X|\le r-k$, and $r_i(X)=r$ holds if $|X|\ge r+k$.
Then, a maximum weight common independent set of $\mathbf{M}_1$ and $\mathbf{M}_2$ can be computed in time polynomial in $|S|^k$ and $\log B$ in the minimum rank oracle model.
\end{thm}

\section{\textit{W}[1]-hardness}
\label{sec:hardness}

In this section, we prove that $2k+1$-distant submodular function minimization includes the $k$-clique problem, which is \textit{W}[1]-hard when parameterized by $k$ \cite{downey1995}.
Let $G=(V,E)$ be an undirected graph and $k$ a positive integer.
The $k$-clique problem asks whether there exists a vertex set $X\subseteq V$ such that the subgraph $G[X]$ induced by $X$ is a $k$-clique.
Let $f:2^V\to \mathbb{Z}$ be a set function defined as follows:
\begin{align*}
f(X)=\begin{cases}
    -1 & \mathrm{if\ }G[X]\mathrm{\ is\ a\ }k\text{-}\mathrm{clique},\\
    0 & \mathrm{if\ }|X|\le k\ \mathrm{and\ }G[X]\mathrm{\ is\ not\ a\ }k\text{-}\mathrm{clique},\\
    |V\setminus X| & \mathrm{otherwise.}
\end{cases}
\end{align*}
Since the minimum value of $f$ is $-1$ if and only if $G$ includes a $k$-clique, the minimization of $f$ includes the $k$-clique problem.
We now show that $f$ is $2k+1$-distant submodular as follows.
Take $X,Y\subseteq V$ with $|X\bigtriangleup Y|\ge 2k+1$.
Since $|X\cup Y|\ge |X\bigtriangleup Y|\ge 2k+1$, we may assume without loss of generality that $|X|\ge k+1$.
This implies $f(X)=|V\setminus X|$.
If $f(Y)=|V\setminus Y|$, then 
\begin{align*}
f(X)+f(Y)=|V\setminus X|+|V\setminus Y|=|V\setminus (X\cup Y)|+|V\setminus (X\cap Y)|\ge f(X\cup Y)+f(X\cap Y),
\end{align*}
as required.
Consider the case when $f(Y)=0$. 
Since $G[X\cap Y]$ is not a $k$-clique, we have $f(X\cap Y)=0$.
Hence,
\begin{align*}
f(X)+f(Y)=|V\setminus X|\ge |V\setminus (X\cup Y)|=f(X\cup Y)+f(X\cap Y),
\end{align*}
as required.
Consider the case when  $f(Y)=-1$.
If $Y\subseteq X$, then since $X\cup Y=X$ and $X\cap Y=Y$ we have $f(X)+f(Y)=f(X\cup Y)+f(X\cap Y)$.
If $Y\not\subseteq X$, then since $|X\cap Y|<|Y|=k$ and $|X\cup Y|>|X|$, we have
\begin{align*}
f(X)+f(Y)=|V\setminus X|-1\ge |V\setminus (X\cup Y)|=f(X\cup Y)+f(X\cap Y),
\end{align*}
as required.
Hence, $f$ is $2k+1$-distant submodular.
Under the standard parameterized complexity hypothesis $W[1]\ne FPT$, the $k$-clique problem has no $g(k)\cdot (|V|+|E|)^{O(1)}$-time algorithm (FPT algorithm) for any computable function $g:\mathbb{N}\to \mathbb{N}$, which implies that the $k$-distant submodular function minimization problem with a ground set $S$ has no $g(k)\cdot |S|^{O(1)}$-time algorithm for any computable function $g$.

\section*{Acknowledgements}
We are grateful to Satoru Fujishige for his suggestion on $k$-distant submodular functions, and to Yutaro Yamaguchi and Yu Yokoi for informing us of the results of weighted matroid intersection under the minimum rank oracle.
This work was supported by Grant-in-Aid for JSPS Fellows Grant Number JP23KJ0379 and JST SPRING Grant Number JPMJSP2108.

\bibliography{main}
\bibliographystyle{plain}
\end{document}